\documentclass[11pt]{article}
\usepackage{amsmath,amsthm,amssymb,tikz,picins,hyperref}
\usetikzlibrary{patterns}
\usepackage[margin=1.3in]{geometry}
\newtheorem{theorem}{Theorem}
\newtheorem{lemma}[theorem]{Lemma}

\newtheorem{claim}[theorem]{Claim}

\def\hhref#1{\href{http://arxiv.org/abs/#1}{arXiv:#1}}         
\hypersetup{
pdfauthor = {Boris Bukh},
pdftitle = {Multidimensional Kruskal-Katona theorem},
pdfsubject = {Mathematics},
pdfkeywords = {Kruskal-Katona, shadows, compression},
pdfstartview = {FitH},
pdfpagemode = {None}}
\author{Boris Bukh\footnote{\texttt{B.Bukh@dpmms.cam.ac.uk}.
Centre for Mathematical Sciences,
Cambridge CB3 0WB, England
and Churchill College, Cambridge CB3 0DS, England.}}
\title{Multidimensional Kruskal--Katona theorem%
\footnote{The paper is in public domain, and is not protected by copyright. The paper
is available at \hhref{1009.2375}}.}
\date{}

\newcommand*{\R}{\mathbb{R}}                                   
\newcommand*{\F}{\mathcal{F}}                                  
\newcommand*{\N}{\mathbb{N}}                                   
\newcommand*{\abs}[1]{\lvert #1\rvert}                         
\newcommand*{\norm}[1]{\lVert #1\rVert}                        
\DeclareMathOperator{\ind}{ind}                                
\DeclareMathOperator{\extr}{extr}                              
\DeclareMathOperator{\mclos}{mclos}                            
\DeclareMathOperator{\vol}{vol}  
\DeclareMathOperator{\area}{area}  
\newcommand*{\eqdef}{\stackrel{\text{\tiny def}}{=}}

\begin{document}
\maketitle

\begin{abstract}
We present a generalization of a version of the Kruskal--Katona theorem due to Lov\'asz.
A shadow of a $d$-tuple $(S_1,\dotsc,S_d)\in \binom{X}{r}^d$ consists of $d$-tuples
$(S_1',\dotsc,S_d')\in \binom{X}{r-1}^d$ obtained by removing one element
from each of the $S_i$. We show that if a family $\F\subset \binom{X}{r}^d$ has size
$\abs{\F}=\binom{x}{r}^d$ for a real number $x\geq r$, then
the shadow of $\F$ has size at least $\binom{x}{r-1}^d$.
\end{abstract}

\section*{Introduction}
An $r$-uniform set family $\F$ is simply a collection of $r$-element sets. The
shadow of $\F$, denoted $\partial \F$, consists of all $(r-1)$-element
sets that can be obtained by removing an element from a set in $\F$. 
If $(X,<)$ is an ordered set, then $A\subset X$
is colexicographically smaller than $B\subset X$ if the largest element of 
$(A\cup B)\setminus(A\cap B)$ lies in $B$. 

The Kruskal--Katona theorem \cite{kruskal,katona} is a classic 
result in combinatorics that states that 
$\abs{\partial \F}\geq \abs{\partial F_0}$,
where $\F_0$ is the initial segment of length 
$\abs{\F}$ in colexicographic order
on $r$-tuples of some ordered set. Moreover equality 
is achieved only if $\F$ is an initial segment of such 
a colexicographic order. As the quantitative form
of the Kruskal--Katona theorem is unwieldy, in applications 
one usually uses the weaker form due to Lov\'asz \cite[Ex.~13.31(b)]{lovasz}: 
if $\abs{\F}=\binom{x}{r}$ for some
real number\footnote{For real $x$ and integer $r$ the binomial 
coefficient $\binom{x}{r}$ is defined by $x(x-1)\dotsb(x-r+1)/r!$.} 
$x\geq r$, then $\abs{\partial F}\geq \binom{x}{r-1}$.

In this paper we present a generalization of Lov\'asz's theorem to 
multidimensional $r$-uniform families. A \emph{$d$-dimensional $r$-uniform 
family} is a collection of $d$-tuples of $r$-element sets. In other words, if we 
denote by $\binom{X}{r}$ the family of all $r$-element subsets of $X$, 
then a $d$-dimensional $r$-uniform family is a subset of $\binom{X}{r}^d$.
A \emph{shadow} of such a family $\F\subset \binom{X}{r}^d$ is defined to be
\begin{equation*}
\partial \F \eqdef \{ (S_1\setminus\{x_i\},\dotsc,S_d\setminus\{x_d\} ) :  
   (S_1,\dotsc,S_d) \in \F, \text{and }x_i\in S_i\text{ for }i=1,\dotsc,d\}.
\end{equation*}
The special case $d=1$ of the following theorem is Lov\'asz's result.
\begin{theorem}\label{main_theorem}
Suppose $\F\subset \binom{X}{r}^d$ is a $d$-dimensional $r$-uniform family
of size
\begin{equation*}
\abs{\F}=\binom{x}{r}^d,
\end{equation*}
where $x\geq r$ is a real number. Then
\begin{equation*}
\abs{\partial \F}\geq \binom{x}{r-1}^d.
\end{equation*}
Moreover, equality holds only if $\F$ is of the form
$\binom{Y_1}{r}\times \dotsb\times \binom{Y_d}{r}$ for some sets
$Y_1,\dotsc,Y_d\subset X$.
\end{theorem}
The rest of the paper contains the proof of this result.

\section*{Proof}
For simplicity of notation we shall assume that the ground set 
is $[n]\eqdef \{1,2,\dotsc,n\}$, with the ordering on it being the 
standard ordering of the integers. This incurs no loss
of generality.

A $k$-dimensional \emph{section} of a $d$-dimensional family $\F\subset \binom{X}{r}^d$
is the subfamily of $\F$ obtained by fixing $d-k$ coordinates. For example, for
any $(d-k)$-tuple $S=(S_1,\dotsc,S_{d-k})\in \binom{X}{r}^{d-k}$ the family
\begin{equation*}
\F_S\eqdef \bigl\{ (S_{d-k+1},\dotsc,S_d)\in \binom{X}{r}^k : (S_1,\dotsc,S_d)\in \F \bigr\}
\end{equation*}
is a $k$-section of $\F$. In general, any $d-k$ coordinates might be fixed, not necessarily
the first $d-k$.

We say that a family $\F\subset \binom{X}{r}^d$ is \emph{monotone} if every $1$-dimensional 
section is an initial segment in the colexicographic order.
\begin{lemma}[Proof deferred to p.~\pageref{compression_lemma_proof}]\label{compression_lemma}
For every family $\F\subset \binom{[n]}{r}^d$ there is a monotone family 
$\F_0\subset \binom{[n]}{r}^d$ of the same size as $\F$, and such that
$\abs{\partial \F_0}\leq \abs{\partial \F}$.
\end{lemma}

By the Lemma~\ref{compression_lemma} it suffices to restrict the attention to monotone families.
The shadows of monotone families are most easily described using the colexicographic 
ordering. This will permit us to establish a correspondence between $d$-dimensional
monotone families and subsets of $\N^d$. Let $\N\eqdef \{1,2,\dotsc\}$
be the set of positive integers, and partially order 
$\N^d$ by 
\begin{equation}\label{prodorder}
(x_1,\dotsc,x_d)\leq (y_1,\dotsc,y_d)\text{ whenever } x_i\leq y_i\text{ for every }i=1,\dotsc,d.
\end{equation}
A set $L\subset \N^d$ is said
to be \emph{monotone} if whenever $x=(x_1,\dotsc,x_d)\in L$, then $L$
contains all the elements smaller than $x$.

If $S\in \binom{[n]}{r}$ is the $i$'th
in the colexicographic ordering on $\binom{[n]}{r}$, then we put $\ind_r(S)=i$. A
tuple $S=(S_1,\dots,S_d)\in \binom{[n]}{r}^d$ is mapped to 
$\ind_r(S)\eqdef (\ind_r(S_1),\dotsc,\ind_r(S_d))$. In this manner every 
$\F\subset \binom{[n]}{r}^d$ is associated with its image $\ind_r(\F)\subset \N^d$.
An \emph{extreme point} of a monotone 
set $L\subset \N^d$ is a point $x\in L$ such that no point in $L$ 
is larger than $x$. The set of extreme points
of $L$ will be denoted $\extr L$. 
The \emph{monotone closure} of a set $L\subset \N^d$
is the set $\mclos(L)=\{x\in \N^d : x\leq y\text{ for some }y\in L\}$.
It is clear that $L=\mclos \extr L$ for any finite set $L$.

For an integer $m\geq 1$ let $KK_r(m)$ be the size of a shadow 
of the initial segment of length $m$ in colexicographic order 
of $\binom{[n]}{r}$. The Kruskal--Katona theorem states that if $\F\subset \binom{[n]}{r}$,
then $\abs{\partial \F}\geq KK_r(\abs{\F})$. 
We extend the definition of $KK_r$ to $KK_r\colon \N^d\to\N^d$ by 
$KK_r(a_1,\dotsc,a_d)\eqdef (KK_r(a_1),\dotsc,KK_r(a_d))$.
\begin{lemma}[Proof deferred to p.~\pageref{lemma_monotone_shadow_proof}]\label{lemma_monotone_shadow}
Let $\F\subset\binom{[n]}{r}^d$ be a monotone family. 
Then its shadow $\partial \F$ is also a monotone family, and 
\begin{equation*}
\extr \ind_{r-1} (\partial \F)=KK_r(\extr \ind_r (\F)).
\end{equation*}
\end{lemma}
The preceding lemma permits us to forget about shadows of set 
families, and instead think about images of monotone sets under
$KK_r$. However, as $KK_r$ is quite an erratic function, our next step is
to replace it by a smoother function. For an integer $r\geq 2$ put 
\begin{equation}\label{llfup}
LL_r\Bigl(\binom{x}{r}\Bigr)=\binom{x}{r-1}\qquad\text{ if }x\geq r.
\end{equation}
Since $\binom{x}{r}$ is an increasing function of $x$ for $x\geq r-1$,
the function $LL_r$ is well-defined on $[1,\infty)$. 
We would like to extend $LL_r$ to $[0,1)$ while
maintaining the inequality $LL_r\leq KK_r$. Furthermore, as it will become
clear below, it will be essential for $LL_r$ to be increasing, concave and
to satisfy
\begin{equation}\label{superlogconc}
x\frac{f'(x)}{f(x)}<y\frac{f'(y)}{f(y)}\qquad\text{when }x>y.
\end{equation} 
Any extension of $LL_r$ to $[0,\infty)$ satisfying these conditions
is equally good for us. For example, one permissible extension is
\begin{equation}\label{llfdown}
LL_r(x)=r\left(x+\frac{1}{\sum_{i=1}^r 1/i}(x-x^2)\right)\qquad\text{ if }0\leq x\leq 1.
\end{equation}
\begin{lemma}[Proof deferred to p.~\pageref{lemma_llr_check_proof}]\label{lemma_llr_check}
The function $LL_r$ defined by \eqref{llfup}
and \eqref{llfdown} is a continuously differentiable function that 
is strictly increasing, concave, and satisfies \eqref{superlogconc}.
\end{lemma}

Extend $LL_r$ to $LL_r\colon \R_+^d\to\R_+^d$ by $LL_r(x_1,\dotsc,x_d)\eqdef
(LL_r(x_1),\dotsc,LL_r(x_d))$. Put $\R_+=[0,\infty)$. 
Partially order $\R_+^d$ according to \eqref{prodorder}, and extend
the definitions of the terms ``monotone'' and ``extreme point'' in the obvious way.
We associate to every monotone set $L\subset \N^d$ the set $M\subset \R_+^d$
given by $M=L+[-1,0]^d$. Geometrically, $M$ is the set obtaining by filling 
in the square lattice boxes indexed by $L$. The volume of $M$ is 
equal to the number of points in $L$. The set $M$ so obtained is monotone.
Since $LL_r(0)=0$ and $LL_r\leq KK_r$, Lemma~\ref{lemma_monotone_shadow} implies that
if $\abs{\partial \F} \leq t$ for some family $\F\subset \binom{[n]}{r}^d$,
then there is a closed monotone set $M\subset \R_+^d$ with $\vol(M)=\abs{\F}$ for which
$\vol(LL_r(M))\leq t$. The Theorem~\ref{main_theorem} thus follows from the following claim.
\begin{claim}\label{claim_final}
Suppose $f\colon \R_+\to\R_+$ is a continuously differentiable, strictly increasing, concave function 
satisfying \eqref{superlogconc} and $f(0)=0$. Define $f\colon \R_+^d\to \R_+^d$ by
$f(x_1,\dotsc,x_d)=(f(x_1),\dotsc,f(x_d))$. Then for every closed 
monotone set $M\subset \R_+^d$ we have
\begin{equation*}
\vol(f(M))\geq \vol(f(M_0))
\end{equation*}
where $M_0=[0,\sqrt[d]{\vol(M)}]^d$ is the cube of the same
volume as $M$, and one of whose vertices is at the origin.
Furthermore equality holds only if $M=M_0$.
\end{claim}

To prove the claim we shall first establish it in the dimension $d=2$, and 
use that to deduce the general case. Indeed, assume that the two-dimensional
case is known, $d\geq 3$, and $M$ is not a cube. Pick any $2$-dimensional
coordinate plane $P$. On each $2$-dimensional section
of $M$ by a plane parallel to $P$, replace the section of $M$ 
by a square of the same area as the area of that section. The operation
yields a monotone set, and by the case $d=2$ of the claim, 
it reduces the volume of $f(M)$ unless every section of $M$ 
is a square. Therefore, the only minimizer of $\vol(f(M))$ 
is the cube $[0,\sqrt[d]{\vol(M)}]^d$.

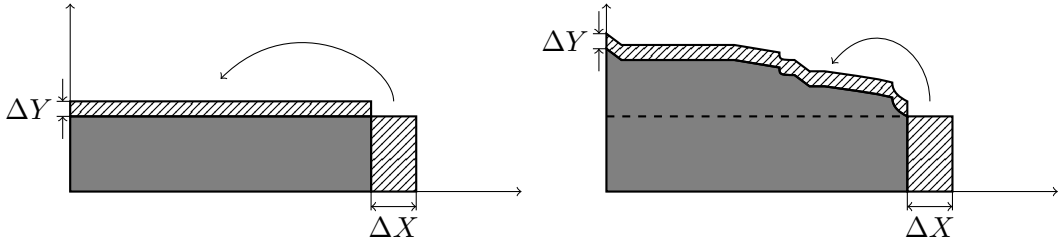
\begin{figure}[t]
\begin{tikzpicture}
\draw [<->] (6,0) -- (0,0) -- (0,2.5);
\filldraw[draw=black,fill=gray,thick] (0,0) rectangle (4,1);
\draw[thick,pattern=north east lines,pattern color=black] (4,0) rectangle (4.6,1);
\draw[thick,pattern=north east lines,pattern color=black] (0,1) rectangle (4,1.2);
\draw [->] (4.3,1.2) .. controls (4.3,1.7) and (3,2.5) .. (2,1.5);
\draw[thin] (4,0) -- (4,-0.25);
\draw[thin] (4.6,0) -- (4.6,-0.25);
\draw[<->, thin] (4,-0.2) -- (4.3,-0.2) node [anchor=north] {$\Delta X$} -- (4.6,-0.2);
\draw[thin] (0,1) -- (-0.18,1);
\draw[thin] (0,1.2) -- (-0.18,1.2);
\draw[->, thin] (-0.1,1.5) -- (-0.1,1.2);
\draw[->, thin] (-0.1,0.7) -- (-0.1,1);
\draw (-0.1,1.1) node [anchor=east] {$\Delta Y$};
\end{tikzpicture}
\begin{tikzpicture}
\draw [<->] (6,0) -- (0,0) -- (0,2.5);
\filldraw[draw=black,fill=gray,thick] (0,0) -- (4,0) -- (4,1) 
.. controls (3.9,1.05) and (3.8,1.12) .. (3.8,1.25) .. controls (3.6,1.3) .. (2.9,1.4)
-- (2.7,1.4) -- (2.5,1.55) .. controls (2.3,1.55) and (2.3,1.56) .. (2.3,1.65) 
-- (1.7,1.75) -- (0.2,1.75) -- (0,1.9) -- cycle;
\draw[thick,pattern=north east lines,pattern color=black] (4,0) rectangle (4.6,1);
\draw[thick, pattern=north east lines, pattern color=black] (4,1.2) 
.. controls (3.9,1.25) and (3.8,1.32) .. (3.8,1.45) .. controls (3.6,1.5) .. (2.9,1.6)
-- (2.7,1.6) -- (2.5,1.75) .. controls (2.3,1.75) and (2.3,1.76) .. (2.3,1.85) 
-- (1.7,1.95) -- (0.2,1.95) -- (0,2.1) -- 
(0,1.9) -- (0.2,1.75) -- (1.7,1.75) -- (2.3,1.65) .. controls (2.3,1.56) and (2.3,1.55) .. (2.5,1.55)
-- (2.7,1.4) -- (2.9,1.4) .. controls (3.6,1.3) .. (3.8,1.25) .. controls (3.8,1.12) and
(3.9,1.05) .. (4,1) -- cycle;
\draw[thick,dashed] (0,1) -- (4,1);
\draw [->] (4.3,1.2) .. controls (4.3,1.9) and (3.55,2.4) .. (3,1.7);
\draw[thin] (4,0) -- (4,-0.25);
\draw[thin] (4.6,0) -- (4.6,-0.25);
\draw[<->, thin] (4,-0.2) -- (4.3,-0.2) node [anchor=north] {$\Delta X$} -- (4.6,-0.2);
\draw[thin] (0,1.9) -- (-0.18,1.9);
\draw[thin] (0,2.1) -- (-0.18,2.1);
\draw[->, thin] (-0.1,2.4) -- (-0.1,2.1);
\draw[->, thin] (-0.1,1.6) -- (-0.1,1.9);
\draw (-0.1,2) node [anchor=east] {$\Delta Y$};
\end{tikzpicture}
\caption{\label{compfig}The area-reducing transformation for an elongated rectangle (left),
and for a general monotone set (right).}
\end{figure}

So assume $d=2$.
To see where the condition \eqref{superlogconc} comes from consider the case where $M$
is a rectangle, i.e.\ a set of the form $M=[0,X]\times [0,Y]$, with say $X>Y$. In that case,
if we are to move a small amount of mass from the shorter side to the longer one,
to obtain a less elongated rectangle $M^*=[0,X-\Delta X]\times [0,Y+\Delta Y]$, then
\eqref{superlogconc} is exactly what is necessary to conclude that
$\area(f(M^*))<\area(f(M))$.

The situation when $M$ is not a rectangle is to our advantage because
$f$ is concave and we place the mass farther from the origin than
in the case when $M$ is a rectangle. The only complication is that we need
to introduce continuous time to avoid technicalities arising from discrete
time increments.

Since $M$ is monotone there is a decreasing function $g_{\infty}\colon \R_+\to\R_+$
so that $M=\{(x,y)\in \R_+^2 : y\leq g_{\infty}(x)\}$. Since $M$ is closed, $g_{\infty}$ is
left-continuous. 
Define $g_t\colon \R_+\to \R_+$ by
\begin{equation*}
g_t(x)=
\begin{cases}
g_\infty(x)+\frac{1}{t}\int_{[t,\infty)}g_\infty(y)\,dy&\text{if }x\leq t,\\
0&\text{if }x>t.
\end{cases}
\end{equation*}
Let $M_t=\{(x,y)\in\R_+^2 : y\leq g_t(x)\}$. Then $\area(M_t)=\area(M)$.
Differentiating 
\begin{equation*}
\area(f(M_t))=\int_{[0,t]}f(g_t(x))f'(x)\,dx,
\end{equation*}
we obtain
\begin{align*}
\frac{\partial \area(f(M_t))}{\partial t}&=
f(g_t(t))f'(t)+\int_{[0,t]}f'(g_t(x))\frac{\partial g_t}{\partial t}(x)f'(x)\,dx\\
&\geq f(g_t(t))f'(t)+f'(g_t(t))\int_{[0,t]}\frac{\partial g_t}{\partial t}(x)f'(x)\,dx\\
&= f(g_t(t))f'(t)+f'(g_t(t))\left(\frac{\partial g_t}{\partial t}(t)f(t)-\int_{[0,t]}f(x)
\frac{\partial^2 g_t(x)}{\partial x\partial t}\,dx\right)
\\&=f(g_t(t))f'(t)+f'(g_t(t))\frac{\partial g_t}{\partial t}(t)f(t),
\end{align*}
where the inequality holds since $f$ is concave, and $(\partial g_t/\partial t)f'$ is negative (see
Figure~\ref{compfig} for a geometric illustration of the inequality).
Since $\partial g_t/\partial t\geq -g_t(t)/t$, from \eqref{superlogconc} it follows
that $\area(f(M_t))$ is an increasing function of $t$ as long as $g_t(t)<t$.

Let $T=\sqrt{\area(M)}$.
Since $\area(M_t)\geq t g_t(t)$, it follows that  
$g_t<t$ for every $t>T$. Thus $\area(f(M_T))\leq \area(f(M))$, with
equality only if $M\subset [0,T]\times \R_+$. Since 
$g_T(x)\leq g_\infty(x)+\area(M)/T$ it follows that if $M\subset \R_+\times [0,Y]$,
then $M_T\subset [0,T]\times [0,Y+\area(M)/T]=[0,T]\times [0,Y+T]$.
Reversing the roles of $x$ and $y$ axes, and applying the argument to $M_T$, 
it follows that for every closed monotone set $M\subset \R_+^2$ there is a 
compact monotone set $M'\subset [0,2T]\times [0,T]$ for which
$\area(f(M'))\leq \area(f(M))$ with equality holding only for $M=[0,T]^2$.
Since the space of compact monotone subsets of $[0,2T]\times [0,T]$
endowed with Hausdorff distance is a compact space, and
$\area(f(\cdot))$ is a continuous function on the space, it follows
that $[0,T]^2$ is a unique set minimizing this function.
This completes the proof of the Claim~\ref{claim_final} in the 
case $d=2$.

\section*{Deferred lemmas}
\begin{proof}[Proof of Lemma~\ref{compression_lemma}]\label{compression_lemma_proof}
For the duration of this proof define the \emph{weight} of $\F\subset \binom{[n]}{r}^d$
to be $\sum_{S\in \F} \norm{\ind_r(\F)}_1$, where 
$\norm{(m_1,\dotsc,m_d)}_1=m_1+\dotsb+m_d$. We may assume that $\F$ has
smallest weight among families of size $\abs{\F}$ and whose shadow
does not exceed $\abs{\partial \F}$.

Suppose some $1$-dimensional section of $\F$ is not an initial segment
of the colexicographic order. Without loss of generality we may
assume that the section is of the form $\F_S$ for some $S$.
Define a compression operator $\Delta\colon 2^{\binom{[n]}{r}}\to 2^{\binom{[n]}{r}}$
which takes $\F\subset \binom{[n]}{r}$ to the initial segment
of $\binom{[n]}{r}$ in the colexicographic order. One can write $\F$
as a disjoint union of its $1$-dimensional sections as 
\begin{equation*}
\F=\bigcup_{S\in \binom{[n]}{r}^{d-1}} \{S\} \times \F_S.
\end{equation*}
Define
\begin{equation*}
\F'=\bigcup_{S\in\binom{[n]}{r}^{d-1}} \{S\} \times \Delta \F_S.
\end{equation*}
We claim that $\abs{\partial \F'}\leq \abs{\partial \F}$. Indeed,
let $S'\in \binom{[n]}{r-1}^{d-1}$ be arbitrary, and consider
the section $(\partial \F')_{S'}$. The section has at least 
$t$ elements if and only if there is a $S\in \binom{[n]}{r}^{d-1}$
such that $S'\in \partial S$ and $KK_r(\abs{\F_S})\geq t$.
Hence, if $\abs{(\partial \F')_{S'}}\geq t$, then by the classical 
Kruskal--Katona inequality $\abs{(\partial \F)_{S'}}\geq t$.
Since the inequality holds for every $S'$, it follows that
\begin{equation*}
\abs{\partial \F}=\sum_{S'\in \binom{[n]}{r-1}^{d-1}} \left\lvert(\partial \F)_{S'}\right\rvert
\geq \sum_{S'\in \binom{[n]}{r-1}^{d-1}} \left\lvert(\partial \F')_{S'}\right\rvert=\abs{\partial \F'}
\end{equation*}
Since the weight of $\F'$ is less than that of $\F$, this contradicts the choice of $\F$.
\end{proof}

\begin{proof}[Proof of Lemma~\ref{lemma_monotone_shadow}]\label{lemma_monotone_shadow_proof}
First we establish that $\partial \F$ is monotone. Suppose 
$S=(S_1,S_2,\dotsc,S_d)\in \partial \F$  and $S_1'$ precedes $S_1$ 
in colexicographic order. There is an 
$\bar{S}=(\bar{S}_1,\dotsc,\bar{S}_d)\in \F$ so that $S\in \partial \bar{S}$.
Since the shadow of an initial segment of colexicographic order
is an initial segment of colexicographic order, there is
an $\bar{S}_1'\in \binom{[n]}{r}$ so that $\bar{S}_1'$ precedes
$\bar{S}_1$ in the order, and $S_1'\in\partial \bar{S}_1'$.
Thus $(S_1',S_2,\dotsc,S_d)\in \partial (\bar{S}_1',S_2,\dotsc,S_d)
\subset \partial \F$. This shows that the $1$-dimensional section
 $(\partial \F)_{S_2,\dotsc,S_d}$ of $\partial \F$ is monotone.
Since ordering of coordinates is arbitrary, it follows that every $1$-dimensional
section of $\F$ is monotone, i.e.\ $\F$ is monotone.

From the definition of $KK_r$ it follows that 
$\max \ind_{r-1}(\partial \F_0)=KK_r(\abs{F}_0)$
whenever $\F_0$ is the initial segment of $\binom{[n]}{r}$
in the colexicographic order.  The second claim of the Lemma 
is then again a consequence of the fact
that an image of an initial segment of colexicographical order on $\binom{[n]}{r}$ is 
an initial segment on $\binom{[n]}{r-1}$. 
\end{proof}

\begin{proof}[Proof of Lemma~\ref{lemma_llr_check}]\label{lemma_llr_check_proof}
It is clear that the function defined by \eqref{llfup}
is a continuous monotone increasing function.
The concavity of $LL_r$ on $(1,\infty)$ follows from a simple derivative calculation:
Indeed, for $x\geq r$
\begin{align*}
\frac{d}{dx} LL_r\Bigl(\binom{x}{r}\Bigr)&=\frac{d}{dx} \binom{x}{r-1},\\
LL_r'\Bigl(\binom{x}{r}\Bigr)\binom{x}{r}\left(\frac{1}{x}+\dotsb+\frac{1}{x-r+1}\right)&=
\binom{x}{r-1}\left(\frac{1}{x}+\dotsb+\frac{1}{x-r+2}\right),\\
1/LL_r'\Bigl(\binom{x}{r}\Bigr)&=
\frac{x-r+1}{r} \frac{\frac{1}{x}+\dotsb+\frac{1}{x-r+1}}{\frac{1}{x}+\dotsb+\frac{1}{x-r+2}},\\
1/LL_r'\Bigl(\binom{x}{r}\Bigr)&=
\frac{1}{r}\left(x-r+1+\frac{1}{\frac{1}{x}+\dotsb+\frac{1}{x-r+2}}\right),
\end{align*}
from which it is clear that $LL_r'$ is decreasing on $(1,\infty)$.
Moreover this expression for $LL_r'$ and
\begin{equation*}
LL\Bigl(\binom{x}{r}\Bigr)/\binom{x}{r}=\binom{x}{r-1}/\binom{x}{r}=r/(x-r+1)
\end{equation*}
imply that
\begin{equation*}
\frac{LL_r\Bigl(\binom{x}{r}\Bigr)}{\binom{x}{r} LL_r'\Bigl(\binom{x}{r}\Bigr)}=
1+\frac{1}{(x-r+1)\left(\frac{1}{x}+\dotsb+\frac{1}{x-r+2}\right)}.
\end{equation*}
Since $(x-r+1)/(x-t)$ is a decreasing function of $x$ for every $t<r-1$,
it follows that $\frac{LL_r\Bigl(\binom{x}{r}\Bigr)}{\binom{x}{r} LL_r'\Bigl(\binom{x}{r}\Bigr)}$
is increasing, i.e. $LL_r$ satisfies \eqref{superlogconc} on $(1,\infty)$.

Since $x-x^2$ is concave, the function given by \eqref{llfdown} is concave on $[0,1)$.
For brevity of notation put $\epsilon\eqdef \left(\sum_{i=1}^r 1/i\right)^{-1}$.
Monotonicity of $LL_r$ on $[0,1)$ follows from $\epsilon<1$. Furthermore, for $x\in[0,1)$
we have
\begin{equation*}
x\frac{LL_r'(x)}{LL_r(x)}=x\frac{1+\epsilon(1-2x)}{x+\epsilon(x-x^2)}=
2-\frac{1+\epsilon}{1+\epsilon(1-x)},
\end{equation*}
from which we see that $LL_r$ satisfies \eqref{superlogconc} on $[0,1)$.
Finally, it is easy to check that at $x=1$ the function $LL_r(x)$ is continuous and
the left and right derivatives agree.
\end{proof}

\section*{Concluding remarks}
For us the original motivation for the study of shadows of $d$-dimensional families
was in their application to convexity spaces, and Eckhoff's conjecture\cite{bukh_eckhoff}. For that
application Theorem~\ref{main_theorem} sufficed. However, it would be interesting to
find the sharp multidimensional generalization of Kruskal--Katona theorem.

It is worth noting that the argument given in this paper is largely insensitive to the poset structure
of $2^X$. The only input it uses is the one-dimensional Kruskal--Katona theorem. First,
Lemma~\ref{compression_lemma} is a direct consequence of the fact that the Kruskal--Katona 
theorem equality is attained only for an initial segment of a certain linear order. 
Secondly, a weaker quantitative form of the Kruskal--Katona theorem is used in Lemma~\ref{lemma_llr_check} 
to construct a continuous function to which Claim~\ref{claim_final} 
applies.\vspace{1.5ex}

\noindent\textbf{Acknowledgement.} I am thankful to Peter Keevash for a helpful discussion.
I thank the referee for many improvements in presentation.

\bibliographystyle{alpha}
\bibliography{kkmultidim}

\end{document}